\documentclass[a4paper,11pt]{amsart}
\usepackage{amsfonts}
\usepackage{amstext}
\usepackage{amsthm}
\usepackage{amsmath}
\usepackage{amssymb}
\usepackage{mathtools}
\usepackage{latexsym}
\usepackage{graphicx}
\usepackage{caption}
\usepackage{enumitem}
\usepackage[utf8]{inputenc}
\usepackage[hidelinks]{hyperref}
\usepackage{rotating}
\usepackage{xcolor}
\usepackage{mathrsfs}
\usepackage[a4paper,margin=2.3truecm]{geometry}

\newcommand{\R}{\mathbb{R}}
\newcommand{\N}{\mathbb{N}}

\newcommand{\s}{\mathbb{S}}

\renewcommand{\O}{\Omega}

\newcommand{\be}{\begin{equation}}
\newcommand{\ee}{\end{equation}}
\newcommand{\HH}{\mathdutchcal{H}}

\theoremstyle{plain}

\newtheorem{teor}{Theorem}[section]
\newtheorem{propo}[teor]{Proposition}
\newtheorem{defi}[teor]{Definition}
\newtheorem{lema}[teor]{Lemma}
\newtheorem{cor}[teor]{Corollary}
\theoremstyle{definition}
\newtheorem{rem}[teor]{Remark}
\newtheorem{example}[teor]{Example}

\DeclareMathOperator{\dive}{div}

\DeclareMathAlphabet{\mathdutchcal}{U}{dutchcal}{m}{n}

\title[fundamental frequency and torsional rigidity for anisotropic energies]{On the relations between fundamental frequency and torsional rigidity in the case of anisotropic energies}

\author{\sc Giuseppe Buttazzo\quad and\quad Raul Fernandes Horta}

\date{}
\begin{document}
\maketitle

\begin{abstract}
We consider variational energies of the form
\be\label{energy}
E_H(u)=\frac12\int_\O H^2(\nabla u)\,dx
\ee
defined on the Sobolev space $H^1_0(\O)$, where $H$ is a general seminorm. Our primary objective is to investigate optimization problems associated with the first eigenvalue  $\lambda_H(\O)$ and the torsional rigidity $T_H(\O)$ induced by the seminorm $H$. In particular, we focus on functionals of the type
\[F_{q,\O}(H)=\lambda_H(\O)\,T_H^q(\O),\]
where $q>0$ is a fixed real parameter. The optimization is performed with respect to the control $H$; we analyze both minimization and maximization problems for $F_{q,\O}(H)$, as $H$ ranges over a suitable class of seminorms.
\end{abstract}

\textbf{Keywords:} Shape optimization, spectral optimization, anisotropic energies, anisotropic torsional rigidity, anisotropic eigenvalues.

\textbf{2020 Mathematics Subject Classification:} 49J20, 49J30, 49J45, 47J20.

\section{Introduction}\label{s_intro}

In this paper, we investigate a family of variational energies of the form
\[E_H(u)=\frac12\int_\O H^2(\nabla u)\,dx\]
defined on the Sobolev space $H^1_0(\O)$, where $\O$ is a prescribed bounded open subset of $\R^d$ and $H$ is a {\it seminorm}, that is a nonnegative 1-homogeneous function on $\R^d$. Such energies naturally arise in the study of anisotropic diffusion and geometric variational problems, where the function $H$ encodes directional or structural preferences of the underlying medium.

Associated with the functional $E_H$, we define two fundamental quantities: the {\it first eigenvalue} and the {\it torsional rigidity} (or shortly, the {\it torsion}). The first eigenvalue is given by
\be\label{eigen}
\lambda_{H}(\O) = \inf_{u \in H^{1}_{0}(\O)\setminus\{0\} } \frac{\int_{\O} H^{2}(\nabla u) \ dx}{\int_{\O} u^2\ dx},
\ee
while the torsional rigidity is defined by
\be\label{torsion}
T_{H}(\O) = \sup_{u\in H^1_0(\O)\setminus\{0\}}\frac{\left(\int_\O u\ dx\right)^2}{\int_{\O} H^{2}(\nabla u) \ dx}.
\ee
These quantities play a central role in anisotropic analysis, as they describe, respectively, the fundamental frequency and compliance of the domain $\O$ under the metric induced by $H$. Indeed, the corresponding PDEs are
\[\begin{cases}
-\dive\big(H(\nabla u)\nabla H(\nabla u)\big)=\lambda u&\text{in }\O\\
u=0&\text{on }\partial\O
\end{cases}\qquad\qquad\text{for the eigenvalue,}\]
\[\begin{cases}
-\dive\big(H(\nabla u)\nabla H(\nabla u)\big)=1&\text{in }\O\\
u=0&\text{on }\partial\O
\end{cases}\qquad\qquad\text{for the torsion.}\]

Our principal goal is to study optimization problems involving the quantities $\lambda_H$ and $T_H$, where the main control variable is the seminorm $H$. To this end, we define
\begin{equation}\label{eqnorm}
\|H\|=\sup\bigg\{\frac{H(\xi)}{|\xi|}\ :\ \xi\in\R^d\setminus\{0\}\bigg\}
\end{equation}
and we consider the space
\[\HH = \left\{H:\R^d\to\R\ |\ H\text{ is a seminorm}\right\}.\]

The problems of minimizing or maximizing $\lambda_H$ and $T_H$ over the class $\HH$, normalized by $\|H\|=1$, have been previously investigated in the literature (see \cite{HM,HM2,HM3,HHM}), and we briefly recall the main results in Section \ref{s_preli}.
In the present work, we focus on a mixed optimization framework involving the cost functional
\[F_{q,\O}(H)=\lambda_H(\O)\,T_H^q(\O),\]
where $q>0$ is a fixed real parameter. Both the minimization and maximization problems for the functional $F_{q,\O}$ within the admissible class $\HH$ are interesting, as they capture different balances between rigidity and flexibility in the anisotropic setting. Observe that, since $\lambda_H$ is an increasing function of $H$, while $T_H$ is decreasing, the two contributions appearing in the functional $F_{q,\O}$ act in competition with each other. As a consequence, the balance between these opposing effects plays a fundamental role in the optimization process. In particular, the exponent $q$ becomes a key parameter in the analysis, as it determines the relative weight of the two terms and thus has a decisive influence on the qualitative and quantitative properties of the optimal choice of $H$.

We also point out that optimization problems characterized by a similar interplay between competing terms have already been studied in the literature. For instance, in \cite{BBG24} and \cite{BBP22}, the optimization is performed with respect to the domain $\O$, rather than the parameter $H$, while in \cite{BCS26} an analogous problem is considered, in which the classical Dirichlet boundary condition is replaced by a Robin boundary condition.

We investigate the existence of optimal seminorms $\HH_{opt}$, and analyze their dependence on the exponent $q$. In particular, we establish conditions under which the optimal seminorm $\HH_{opt}$ is in fact a {\it norm}, thereby revealing geometric and analytical rigidity phenomena underlying the optimal configurations.

In the final section, we provide additional comments and outline some open questions that, in our view, deserve further investigation.

\section{Preliminary results}\label{s_preli}

In this section, we introduce the notation that will be used throughout the remainder of the paper and provide a brief overview of the principal results from the existing literature.

The space $\HH$ is a closed convex subset of the Banach space of $1$-homogeneous functions, the norm of this space is the one given by \eqref{eqnorm}. This space has the Heine-Borel property, i.e. a bounded and closed set is compact. If $H \in \HH$, then the set $\ker{H} = H^{-1}(\{0\})$ is a linear subspace, so it is interesting to consider the following subsets, for $m \in \{0,\dots,d\}$
\[\begin{split}
&\HH_{m} = \left\{H \colon \R^d \to \R\ |\ H\text{ is a seminorm and }\operatorname{codim}\ker{H}=m\right\},\\
&\HH_{\le m} = \bigcup_{k=0}^{m} \HH_{k}.
\end{split}\]
We remark that $\HH_{d}$ is the set of norms, which is a dense open set of $\HH$. We also have that $\HH_{\le m}$ is a closed set. It is trivial that $\HH_{0} = 0$, and for $m = 1$ we have the following characterization
\[
H \in \HH_{1} \ \text{if and only if there is} \ \eta \in \R^d\setminus\{0\} \ \text{such that} \ H(\xi) = |\langle \xi,\eta\rangle| \ \text{for every} \ \xi \in \R^d.
\]
For proofs of this fact see \cite{HHM}. Another interesting subset of $\HH$ that will be investigated is the one of quadratic seminorms, defined by
\[
\mathcal{Q}=\{H\in\HH\ :\ H^2\text{ is a quadratic form}\}.
\]
When considering variational problems, the sphere of $\HH$, defined by
\[
\s(\HH) = \left\{H \in \HH \colon \lVert H \rVert = 1\right\},
\]
will be of interest, as well as the spheres of the subsets were previously defined, denoted by 
\[\begin{split}
&\s(\HH_{m}) = \s(\HH)\cap\HH_{m},\\
&\s(\HH_{\leq m}) = \s(\HH)\cap\HH_{\leq m},\\
&\s(\mathcal{Q}) = \s(\HH)\cap\mathcal{Q}.
\end{split}\]
We also introduce the notation of some important subsets of $\R^d$. For each $d\in\N$ and $a\in\R^d$, we consider the ellipsoid
\[
E_{d}(a) = \left\{x \in \R^{d} \colon \sum_{i=1}^{d}\frac{x_i^2}{a_i^2} < 1\right\},
\]
and we set
\[
B^d = E_{d}((1,\dots,1)) = \left\{\xi \in \R^d \colon |\xi| < 1\right\},\qquad\s^d = \left\{\xi \in \R^d \colon |\xi| =1\right\}.
\]
It is well known that
\[
|B^d| = \omega_d = \frac{\pi^{\frac{d}{2}}}{\Gamma\left(\frac{d+2}{2}\right)},\qquad \mathcal{H}^{d-1}(\s^{d-1}) = d\omega_d. 
\]

\begin{defi}[Generalized First Eigenvalue and Torsion] Let $\O$ be an open set and $H \in \HH$, then we define
\[\begin{split}
&\lambda_{H}(\O) = \inf_{u \in H^{1}_{0}(\O)\setminus\{0\} } \frac{\int_{\O} H^{2}(\nabla u) \ dx}{\int_{\O} u^2\ dx}\\
&T_{H}(\O) = \sup_{u \in H^{1}_{0}(\O)\setminus\{0\} } \frac{\left(\int_{\O} u\ dx\right)^2}{\int_{\O} H^{2}(\nabla u) \ dx}.
\end{split}\]
If $H$ is the Euclidean norm, then we simply write $\lambda_{H}(\O) = \lambda(\O)$ and $T_{H}(\O) = T(\O)$.
\end{defi}

It is well-known that
\[
T(E_{d}(a)) =\frac{\omega_d}{(d+2)}\left(\prod_{i=1}^{d}a_i\right)\left(\sum_{i=1}^{d}\frac{1}{a_i^2}\right)^{-1}.
\]

\begin{propo}\label{propochange} Let $\O$ be a bounded domain, $A$ an invertible matrix and $H_{A}(\xi) = H(A\xi)$. Then
\[\begin{split}
&\lambda_{H_A}(A^T\O) = \lambda_{H}(\O),\\
&T_{H_A}(A^T\O) =  |\det A|T_{H}(\O).
\end{split}\]
\end{propo}

\begin{proof} For any function $u \in H^{1}_{0}(\O)$ we define $u_{A}(x) = u((A^{-1})^Tx)$, so $u_A \in H^{1}_{0}(A^T\O)$ and
\[
\int_{A^T\O} |u_{A}(y)|^2\, dy = |\det A|\int_{\O} |u(x)|^2\, dx.
\]
We have
\[
\int_{A^T\O} H_{A}^{2}(\nabla u_{A}(y))\, dy = \int_{A^T\O} H^2(A A^{-1} \nabla u ((A^{-1})^Tx))\, dx = |\det A|\int_{\O} H^2(\nabla u(x))\,dx,
\]
implying $\lambda_{H}(\O) = \lambda_{H_A}(A^T\O)$. Notice also that
\[
\int_{A^T\O} u_{A}(y)\, dy = |\det A|\int_{\O} u(x)\, dx,
\]
therefore, $T_{H_A}(A^T\O) = |\det A|T_{H}(\O)$.
\end{proof}

The proof for the following proposition can be found in \cite{HM3,HHM}.

\begin{propo}\label{propocomputedegenerate} If $H(\xi,\eta) = G(\eta)$ for every $(\xi,\eta) \in \R^{d-m}\times\R^m$. Then, setting $\O_x=\{y\in\R^m\ :\ (x,y)\in\O\}$ for every $x\in\R^{d-m}$, we have
\[\begin{split}
&\lambda_H(\O)=\inf\Big\{\lambda_G(\O_x)\ :\ x\in\R^{d-m}\Big\}\\
&T_H(\O)=\int_{\R^{d-m}}T_G(\O_x)\,dx.
\end{split}\]
\end{propo}

\begin{cor} Let $H(\xi) = |\xi_d|$, then
\[
\lambda_H\left(\prod_{i=1}^{d}(a_i,b_i)\right) = \frac{\pi^2}{(b_d-a_d)^2}\qquad\text{and}\qquad T_H\left(\prod_{i=1}^{d}(a_i,b_i)\right) = \frac{(b_d-a_d)^3}{12}\prod_{i=1}^{d-1}(b_i-a_i).
\]
\end{cor}

\begin{proof} For the eigenvalue
\[
\lambda_{H}\left(\prod_{i=1}^{d}(a_i,b_i)\right) = \inf_{x \in \R^{d-m}}\lambda\left(\left(\prod_{i=1}^{d}(a_i,b_i)\right)_x\right) = \inf_{x \in \R^{d-m}}\frac{\pi^2}{|(a_d,b_d)|^2} = \frac{\pi^2}{(b_d-a_d)^2}.
\]
Now, for the torsion, we have
\begin{align*}
T_{H}\left(\prod_{i=1}^{d}(a_i,b_i)\right) = \int_{\prod_{i=1}^{d-1}(a_i,b_i)}T(a_d,b_d) \ dx = \frac{(b_d-a_d)^3}{12}\prod_{i=1}^{d-1}(b_i-a_i),
\end{align*}
as required.
\end{proof}

\begin{propo}\label{propoellipsoid}
Let $H_{v}(\xi)=|\langle\xi,v\rangle|$ with $v\in\R^d\setminus\{0\}$ and let $a = (a_1,\dots,a_d)$ with $a_i > 0$ for every $i \in \{1,\dots,d\}$. Then,
\[\begin{split}
&\lambda_{H_v}(E_d(a)) = \frac{\pi^2}{4}\left(\sum_{i=1}^{d}\frac{v_i^2}{a_i^2}\right),\\
&T_{H_v}(E_d(a)) = \frac{\omega_d}{(d+2)}\left(\prod_{i=1}^{d}a_i\right)\left(\sum_{i=1}^{d}\frac{v_i^2}{a_i^2}\right)^{-1}.
\end{split}\]
\end{propo}

\begin{proof} We can assume without loss of generality that $v\in\s^{d-1}$. To simplify notation, write $E_d(a) = E$. There is $M=(m_{ij})_{1\leq i\leq d, 1\leq j\leq d}\in O(d)$ such that $H_{v}(M\xi)=|\xi_d|$; notice that $x\in M^TE$ if and only if $Mx\in E$, i.e.
\[
\sum_{i=1}^{d}\frac{1}{a_i^2}\left(\sum_{j=1}^{d}m_{ij}x_{j}\right)^{2} < 1.
\]
Hence
\[
\left(\sum_{i=1}^{d}\frac{m_{id}^2}{a_{i}^{2}}\right)x_d^2 + 2\left(\sum_{i=1}^{d}\frac{1}{a_i^2}\sum_{j=1}^{d-1}m_{ij}m_{id}x_{j}\right)x_{d}+ \left(\sum_{i=1}^{d}\frac{1}{a_i^2}\left(\sum_{j=1}^{d-1}m_{ij}x_{j}\right)^2-1\right) < 0.
\]
Denote $J = (\delta_{ij})_{1\leq i\leq d, 1\leq j\leq d-1} \in \R^{d\times (d-1)}$. For every $x = (x_1,\dots,x_{d-1}) \in \R^{d-1}$ such that $(M^TE)_x \neq \emptyset$, we have
\begin{align*}
|(M^TE)_x| &= \left(\sum_{i=1}^{d}\frac{m_{id}^2}{a_{i}^{2}}\right)^{-1}\left[4\left(\sum_{i=1}^{d}\frac{m_{id}}{a_{i}^{2}}\sum_{j=1}^{d-1}m_{ij}x_{j}\right)^2-4\left(\sum_{i=1}^{d}\frac{m_{id}^2}{a_{i}^{2}}\right)\left(\sum_{i=1}^{d}\frac{1}{a_{i}^{2}}\left(\sum_{j=1}^{d-1}m_{ij}x_{j}\right)^2-1\right)\right]^{\frac{1}{2}} \\
&= 2\left(\sum_{i=1}^{d}\frac{m_{id}^2}{a_{i}^{2}}\right)^{-1} \left[\left(\sum_{i=1}^{d}\frac{m_{id}^2}{a_{i}^{2}}\right)+\left(\sum_{i=1}^{d}\frac{m_{id}}{a_{i}}\frac{(MJx)_{i}}{a_{i}}\right)^2-\left(\sum_{i=1}^{d}\frac{m_{id}^2}{a_{i}^{2}}\right)\left(\sum_{i=1}^{d}\frac{((MJx)_{i})^2}{a_{i}^{2}}\right)\right]^{\frac{1}{2}}.
\end{align*}
Now, denoting $A = (\delta_{ij} a_{i})_{1\leq i\leq d, 1\leq j\leq d} \in \R^{d\times d}$ and noticing that since $Me_{d} = v$ then $m_{id} = v_i$, we have
\begin{align*}
|(M^TE)_x| &= 2\left(\sum_{i=1}^{d}\frac{v_i^2}{a_i^2}\right)^{-1}\left[\left(\sum_{i=1}^{d}\frac{v_i^2}{a_i^2}\right)+ (\langle A^{-1}v,A^{-1}MJx\rangle)^2 -|A^{-1}v|^2|A^{-1}MJx|^2\right]^{\frac{1}{2}} \\
&= 2\left(\sum_{i=1}^{d}\frac{v_i^2}{a_i^2}\right)^{-\frac{1}{2}}\left[1-\left(|A^{-1}MJx|^2-\frac{(\langle A^{-1}v,A^{-1}MJx\rangle)^2}{|A^{-1}v|^2}\right)\right]^{\frac{1}{2}}\\
&\leq 2\left(\sum_{i=1}^{d}\frac{v_i^2}{a_i^2}\right)^{-\frac{1}{2}},
\end{align*}
where the last line is due to Cauchy-Schwarz inequality, with equality achieved when $x = 0$. Hence
\[\begin{split}
\lambda_{H_{v}}(E)&=\lambda_{H_{v}\circ M}(M^TE) = \inf_{x \in \R^{d-1}}\lambda(0,|(M^TE)_x|)\\
&=\inf_{x \in \R^{d-1}}\frac{\pi^2}{|(M^TE)_x|^2} = \frac{\pi^2}{4}\left(\sum_{i=1}^{d}\frac{v_i^2}{a_i^2}\right).
\end{split}\]
Now notice that the set $E_M = \left\{x \in \R^{d-1}\colon (M^TE)_x \neq \emptyset\right\}$ is an ellipsoid that can be written as
\[
E_M = \left\{x \in \R^{d-1} \colon |A^{-1}MJx|^2-\frac{(\langle A^{-1}v,A^{-1}MJx\rangle)^2}{|A^{-1}v|^2} < 1\right\}.
\]
Let $C = J^TM^TAR^TJ$, where $R \in O(d)$ is such that $R(A^{-1}v) = |A^{-1}v|e_d$ and $Z = (\delta_{id})$. Notice that $\ker{(A^{-1}MZM^TA)} =\{A^{-1}v\}^{\perp}$, hence
\[
A^{-1}MZM^TAy = \frac{\langle A^{-1}v,y\rangle A^{-1}v}{|A^{-1}v|^2}\qquad\text{for every }y\in\R^d.
\]
Therefore, for every $x \in \R^{d-1}$ we have
\[\begin{split}
A^{-1}MJCx&=A^{-1}MJJ^TM^TAR^TJx\\
&=A^{-1}M(I-Z)M^TAR^TJx\\
&=A^{-1}MM^TAR^TJx -A^{-1}MZM^TAR^TJx\\
&=R^TJx-A^{-1}MZM^TAR^TJx\\
&=R^TJx-\frac{\langle A^{-1}v,R^TJx\rangle A^{-1}v}{|A^{-1}v|^2}=R^TJx,
\end{split}\]
which gives
\[\begin{split}
|A^{-1}MJCx|^2-\frac{(\langle A^{-1}v,A^{-1}MJCx\rangle)^2}{|A^{-1}v|^2}&=|R^TJx|^2- \frac{(\langle A^{-1}v,R^TJCx\rangle)^2}{|A^{-1}v|^2}\\
&=|R^TJx|^2 = |Jx|^2 = |x|^2.
\end{split}\]
Hence $C^{-1}E_M=B^{d-1}$, meaning $\omega_{d-1}|\det{C}|=|E_M|$. Therefore
\begin{align*}
T_{H_{v}}(E) &= T_{H_{v}\circ M}(M^TE) = \int_{E_M}\frac{|(M^TE)|_x^3}{12} \ dx \\ 
&=\frac{1}{12}\int_{E_M}8\left(\sum_{i=1}^{d}\frac{v_i^2}{a_i^2}\right)^{-\frac{3}{2}}\left[1-\left(|A^{-1}MJx|^2-\frac{(\langle A^{-1}v,A^{-1}MJx\rangle)^2}{|A^{-1}v|^2}\right)\right]^{\frac{3}{2}} \ dx \\
&=\frac{2}{3}\left(\sum_{i=1}^{d}\frac{v_i^2}{a_i^2}\right)^{-\frac{3}{2}} |\det{C}|\int_{B^{d-1}}(1-|x|^2)^{\frac{3}{2}} \ dx \\ 
&= \frac{2}{3}\left(\sum_{i=1}^{d}\frac{v_i^2}{a_i^2}\right)^{-\frac{3}{2}}|\det{C}|\int_{0}^{1}\int_{\s^{d-2}}(1-r^2)^{\frac{3}{2}}r^{d-2} \ d\omega \ dr\\
&= \left(\sum_{i=1}^{d}\frac{v_i^2}{a_i^2}\right)^{-\frac{3}{2}}|\det{C}|\frac{2(d-1)\omega_{d-1}}{3}\int_{0}^{1}(1-r^2)^{\frac{3}{2}}r^{d-2} \ dr \\
&=\left(\sum_{i=1}^{d}\frac{v_i^2}{a_i^2}\right)^{-\frac{3}{2}}|\det{C}|\frac{2(d-1)\omega_{d-1}}{3}\frac{\Gamma\left(\frac{d-1}{2}\right)\Gamma\left(\frac{5}{2}\right)}{2\Gamma\left(\frac{d+4}{2}\right)}\\
&=\frac{\omega_{d}}{(d+2)\omega_{d-1}}\left(\sum_{i=1}^{d}\frac{v_i^2}{a_i^2}\right)^{-\frac{3}{2}}|E_M|=\frac{1}{(d+2)}\left(\sum_{i=1}^{d}\frac{v_i^2}{a_i^2}\right)^{-1}|E|.
\end{align*}
The last equality is justified as follows: first notice that for every $x\in\R^{d-1}$ we have $ZRAMJx = ce_{d}$ for some $c \in \R$, hence
\[\begin{split}
J^TM^TAR^TZRAMJx&=J^TM^TAR^T(ce_d)\\
&=J^TM^TA(c|A^{-1}v|^{-1}A^{-1}v)\\
&=c|A^{-1}v|^{-1}J ^TM^Tv\\
&=c|A^{-1}v|^{-1}J^Te_d = 0.
\end{split}\]
Therefore $J^TM^TAR^TZRAMJ = 0$, which implies
\begin{align*}
|E_M| &= \omega_{d-1}|\det{C}| = \omega_{d-1} |\det{(J^TM^TAR^TJ)}| \\
&= \omega_{d-1}\sqrt{\det{(J^TM^TAR^TJJ^TRAMJ)}} \\
&= \omega_{d-1}\sqrt{\det{(J^TM^TAR^T(I-Z)RAMJ)}} \\
&= \omega_{d-1}\sqrt{\det{((J^TM^TA^2MJ)-J^TM^TAR^TZRAMJ)}} \\
&= \omega_{d-1}\sqrt{\det{(J^TM^TA^2MJ)}}.
\end{align*}
Since $M^TA^2M$ is positive definite, it admits a Cholesky decomposition as $LL^T$, where $L = (l_{ij})$ is upper triangular, hence
\[
|E_M| = \omega_{d-1}\sqrt{\det{(J^TLL^TJ)}} = \omega_{d-1} \sqrt{\prod_{i=1}^{d-1}l_{ii}^2} = \omega_{d-1}\prod_{i=1}^{d-1} l_{ii}.
\]
On the other hand
\[
|E| = \omega_d|\det{A}| = \omega_d\sqrt{\det(M^TA^2M)} = \omega_d\prod_{i=1}^{d} l_{ii}
\]
and
\[\begin{split}
\left(\sum_{i=1}^{d}\frac{v_i^2}{a_i^2}\right)^{-\frac{1}{2}}&=\frac{1}{|A^{-1}v|}=\frac{1}{|MA^{-1}M^Te_d|}\\
&=\frac{1}{|(M^TAM)^{-1}e_d|} =\frac{1}{\sqrt{\langle (LL^T)^{-1}e_d,e_d\rangle}}\\
&=\frac{1}{|L^{-1}e_d|} = \frac{1}{l_{dd}}.
\end{split}\]
Therefore
\[
\left(\sum_{i=1}^{d}\frac{v_i^2}{a_i^2}\right)^{-\frac{1}{2}}\frac{|E_M|}{\omega_{d-1}} = \prod_{i=1}^{d} l_{ii} = \frac{|E|}{\omega_d}
\]
as needed.
\end{proof}

\begin{cor}\label{corellipsoid} Let $H_{i}(\xi) = |\xi_i|$ and let $a = (a_1,\dots,a_d)$ with $a_i > 0$ for every $i \in \{1,\dots,d\}$. Then,
\[
\lambda_{H_i}(E_d(a)) = \frac{\pi^2}{4a_i^2}\qquad\text{and}\qquad T_{H_i}(E_d(a)) = \left(\prod_{j=1}^{d}a_j\right)\frac{\omega_d a_i^2}{d+2}.
\]    
Consequently, for every $e \in \s^{d-1}$ and $H(\xi) = |\langle \xi,e\rangle|$ we have
\[
\lambda_{H}(B^d) = \frac{\pi^2}{4}\qquad\text{and}\qquad T_{H}(B^d) = \frac{\omega_d}{d+2}.
\]  
\end{cor}

We recall here the results on min/max for $\lambda_H$ and $T_H$, the proofs for $\lambda_{H}$ are given in \cite{HM,HHM} and the proofs for $T_{H}$ are analogous. The isoanisotropic variational constants are defined by
\[\begin{split}
&\lambda_{\min}(\O):=\inf\big\{\lambda_H(\O)\ :\ H\in\s(\HH)\big\},\\
&\lambda_{\max}(\O):=\sup\big\{\lambda_H(\O)\ :\ H\in\s(\HH)\big\},\\
&T_{\min}(\O):=\inf\big\{T_H(\O)\ :\ H\in\s(\HH)\big\},\\
&T_{\max}(\O):=\sup\big\{T_H(\O)\ :\ H\in\s(\HH)\big\}.
\end{split}\]

\begin{defi} Given a bounded open set $\O$, the associated \textbf{directional width function} $L:\s^{n-1}\to\R$, denoted $L_\omega(\O)$, is defined as
\[
L_\omega(\O):=\sup_{v\in\R^n} \sup\Big\{|I|\ :\ I\subset\left\{t\omega + v \colon t \in \R\right\}\cap \O, \ I \ \text{is connected}\Big\}.
\]
\end{defi}

We recall the following results, that can be found on \cite{HM,HHM}.

\begin{teor} Let $\O \subset \R^{d}$ be a bounded domain. Then
\[
\lambda_{\min}(\O):=\inf\big\{\lambda_H(\O)\ :\ H\in\s(\HH_{1})\big\}=\inf_{\omega\in\s^{d-1}}\frac{\pi^2}{L_{\omega}^2(\O)}.
\]
Moreover, the infimum is attained if and only if $\O$ has optimal anisotropic design, that is the function $\omega\mapsto L_\omega(\O)$ has a global maximum. In this case, if $\omega_{0}$ is a global maximum point, the seminorm
\[
H_{0}(\xi) = |\langle\xi,\omega_{0}\rangle|
\]
is an anisotropic extremizer for $\lambda_{\min}(\O)$. Furthermore, if $\partial\O$ is $C^{0,1}$ then all anisotropic extremizers are in $\s(\HH_{1})$. In particular, if $\O$ is convex the minimum is attained and $\omega_0$ is given by the direction of the diameter of $\O$, and
\[
\lambda_{\min}(\O) = \frac{\pi^2}{(\operatorname{diam}(\O))^2}.
\]
\end{teor}

\begin{teor}Let $\O\subset \R^{d}$ be a bounded domain. Then
\[
\lambda_{\max}(\Omega):=\sup\big\{\lambda_H(\O)\ :\ H\in\s(\HH_{d})\big\}=\lambda(\O).
\]
Moreover, if $\partial\Omega$ satisfies the Wiener condition, then the only anisotropic extremizer is the Euclidean norm.
\end{teor}

\begin{teor}Let $\O\subset\R^{d}$ be a bounded domain. Then
\[
T_{\min}(\Omega):=\inf\big\{T_H(\O)\ :\ H\in\s(\HH_{d})\big\} = T(\O).
\]
Moreover, if $\partial\Omega$ satisfies the Wiener condition, then the only anisotropic extremizer is the Euclidean norm.
\end{teor}

\begin{teor}\label{Tmax} Let $\O\subset\R^{d}$ be a bounded domain. Then
\[
T_{\max}(\O):=\sup\big\{T_H(\O)\ :\ H\in\s(\HH_{1})\big\}=\sup_{A\in O(d)}\int_{\R^{d-1}} T\big((A\O)_x\big)\,dx.
\]
Furthermore, if $\partial\O$ is $C^{0,1}$ then all possible anisotropic extremizers are in $\s(\HH_{1})$.
\end{teor}

As a consequence of Theorem~\ref{Tmax} as Proposition~\ref{propoellipsoid} we have the following results for the torsional isoanisotropic problem in ellipsoids.

\begin{cor}\label{corTmax} Let $E$ be an ellipsoid with $a_i$ the length of its semi-major axes. Then,
\[
T_{\max}(E) = \frac{\omega_d}{(d+2)}\left(\prod_{i=1}^{d}a_i\right)  \left(\max_{i \in \{1,\dots,d\}} a_i\right)^2.
\]
Moreover, let $v \in \s^{d-1}$ be a direction such that $L_{v}(E) = \operatorname{diam}(E)$ and $H_v(\xi) = |\langle \xi,v\rangle|$. Then $T_{\max} (E) = T_{H_v}(E)$ and these are the only anisotropic extremizers.
\end{cor}

This result would maybe suggest that the optimal $H \in \s(\HH_1)$ for $T_{\max}(\O)$ would be the one given by direction of the diameter of $\O$. However, as the following example shows, this is not true in general, not even for strictly convex domains.

\begin{example}
Let $\O$ be the right-angle triangle
\[
\O = \left\{(x,y) \in \R^2 \colon 0 < x < 1, 0 <y < 1-x\right\}.
\]
Now let $v_1 = (0,1)$ and $v_2 = \left(\frac{1}{\sqrt{2}},\frac{1}{\sqrt{2}}\right)$ and for each $v$ we consider $H_v$ as in the statement of Corollary~\ref{corTmax}. Notice that if $R_{\theta}$ is the rotation matrix of angle $\theta$, then $H_{v_2}(R_{\frac{\pi}{4}}(x,y)) = |y| = H_{v_1}(x,y)$, also
\[
R_{\frac{\pi}{4}}^T\O = \left\{(x,y) \in \R^2 \colon -\frac{1}{\sqrt{2}} < x < \frac{1}{\sqrt{2}}, |x| < y < \frac{1}{\sqrt{2}}\right\}.
\]
Therefore, from Propositions~\ref{propochange} and~\ref{propocomputedegenerate} we have
\[
T_{H_{v_1}}(\O) = \int_{0}^{1}T(\O_x) \ dx = \int_{0}^{1}\frac{|\O_x|^3}{12} \ dx = \frac{1}{12}\int_{0}^{1}(1-x)^3 \ dx = \frac{1}{48}
\]
and
\begin{align*}
T_{H_{v_2}}(\O) &= T_{H_{v_1}}(R_{\frac{\pi}{4}}^T\O) = \int_{-\frac{1}{\sqrt{2}}}^{\frac{1}{\sqrt{2}}}T((R_{\frac{\pi}{4}}^T\O)_x) \ dx =  \int_{-\frac{1}{\sqrt{2}}}^{\frac{1}{\sqrt{2}}} \frac{\left(\frac{1}{\sqrt{2}}-|x|\right)^3}{12} \ dx \\
&= \frac{1}{6}\int_{0}^{\frac{1}{\sqrt{2}}}\left(\frac{1}{\sqrt{2}}-x\right)^3 \ dx = \frac{1}{6}\cdot\frac{1}{4}\left(\frac{1}{\sqrt{2}}\right)^4 = \frac{1}{96}.
\end{align*}
Hence $L_{v_2}(\O) = \operatorname{diam}(\O)$, but $T_{H_{v_2}}(\O) < T_{H_{v_1}}(\O) \leq T_{\max}(\O)$.
\end{example}

\section{Continuity Properties}\label{scont}

In the present section we study the continuity properties of the maps $H\mapsto\lambda_H(\O)$ and $H\mapsto T_H(\O)$. We start by recalling the following result in \cite{HHM}.

\begin{teor}\label{semicont}Let $\Omega\subset\R^{n}$ be a bounded open domain. The functions 
\[\begin{split}
&H\mapsto T_H(\O)\qquad\text{from $\HH\to\R$}\\
&H\mapsto \lambda_H(\O)\qquad\text{from $\HH\to\R$}\\
\end{split}\]
are lower semicontinuous and upper semicontinuous respectively. Moreover, they are locally Lipschitz in $\HH_d$ and continuous on $0$.
\end{teor}

\begin{teor}\label{layercontinuousC1} Let $\O$ be a $C^1$ domain and $k \in \{0,1,\dots,d\}$. The map
\[
H\mapsto T_H(\O)\qquad\text{from $\HH_k\to\R$}
\]
is continuous.  
\end{teor}

\begin{proof} If $k=0$ then it is obvious and if $k = d$, then it is true because the maps are locally Lipschitz in the set of norms, so assume $k \in \{1,\dots,d-1\}$. Let $H \in \HH_k$, without loss of generality we can assume $H(\xi,\eta) = G(\eta)$ for every $(\xi,\eta) \in \R^{d-k}\times\R^k$, where $G$ is a norm in $\R^k$. Take a sequence $\{H_n\}\subset\HH_k$ converging to $H$. There are rotations $A_n \in O(d)$ and norms $G_n$ in $\R^k$ satisfying $H_n(A_n(\xi,\eta)) = G_n(\eta)$ for every $(\xi,\eta) \in \R^{d-k}\times\R^k$. Clearly, up to a subsequence $G_n$ converges to $G$ and we can assume $A_n$ converges to $I$.

Consider now $x \in \R^{d-k}$ such that for every point in $(\partial\O)_x$, the tangent space is not parallel to $\{x_l = 0\}$ for $l \in \{d-k+1,\dots,d\}$. Since $\O$ is $C^1$. This implies that for $A \in O(d)$ close enough to $I$, $(A\O)_x$ is a finite union of sets like
\[
\{(y,z) \in U_A\times V_A \colon f_{A}(y) < z < g_{A}(y)\},
\]
where the sets $U_A \subset \R^{k-1}$ and $V_A \subset \R$ depend continuously (on the Hausdorff complementary topology) on $A$ and the functions $f_A, g_A \colon U_A \to \R$ have $C^1$ dependency on $A$. Therefore
\begin{equation}\label{convergenceofO}
\lim_{A\to I} (A\O)_x = \O_x,
\end{equation}
where the limit is thought in the Hausdorff complementary topology sense. Since $\O$ is $C^1$, except for a countable set, every $x \in \R^{d-k}$ has the property to ensure \eqref{convergenceofO}. Now notice that, for some compact $K \subset \R^{d-k}$ we have
\begin{align*}
|T_{H_n}(\O) - T_{H}(\O)| &= |T_{H_n\circ A_n}(A_n\O) - T_{H}(\O)| = \left|\int_{K}T_{G_n}((A_n\O)_x)-T_{G}(\O_x) \ dx\right| \\
&\leq \int_{K}|T_{G_n}((A_n\O)_x)-T_{G}((A_n\O)_x)|+|T_{G}((A_n\O)_x)-T_{G}(\O_x)| \ dx \\
&\leq C\lVert G_n - G\rVert^2|K| + \int_{K}|T_{G}((A_n\O)_x)-T_{G}(\O_x)| \ dx.
\end{align*}
The last line we used the locally Lipschitz property of torsion for norms, that is true for $n$ large enough, to complete the proof we use dominated convergence theorem and the continuity of the torsion with respect to the domain in the Hausdorff complementary topology.
\end{proof}

The same result is not true for the map $H \mapsto \lambda_{H}(\O)$ as the following example shows.

\begin{example}\label{edisc} Let $\O$ be the domain in Figure \ref{figure1} and let $H(\xi) = |\langle e,\xi\rangle|$ with $e = (1,0)$. It is easy to see that if $e_n = \left(\cos{\frac{1}{n}},-\sin{\frac{1}{n}}\right)$ and $H_n(\xi) =|\langle e_n, \xi\rangle|$, then
\[
\lim_{n \to \infty} H_{n} = H,
\]
while
\[
\lim_{n \to \infty} \lambda_{H_n}(\O) = \frac{\pi^2}{|C-A|^2} < \frac{\pi^2}{|B-A|} = \lambda_{H}(\O).
\]
\begin{figure}[h]
\centering
\includegraphics[width=0.70\linewidth]{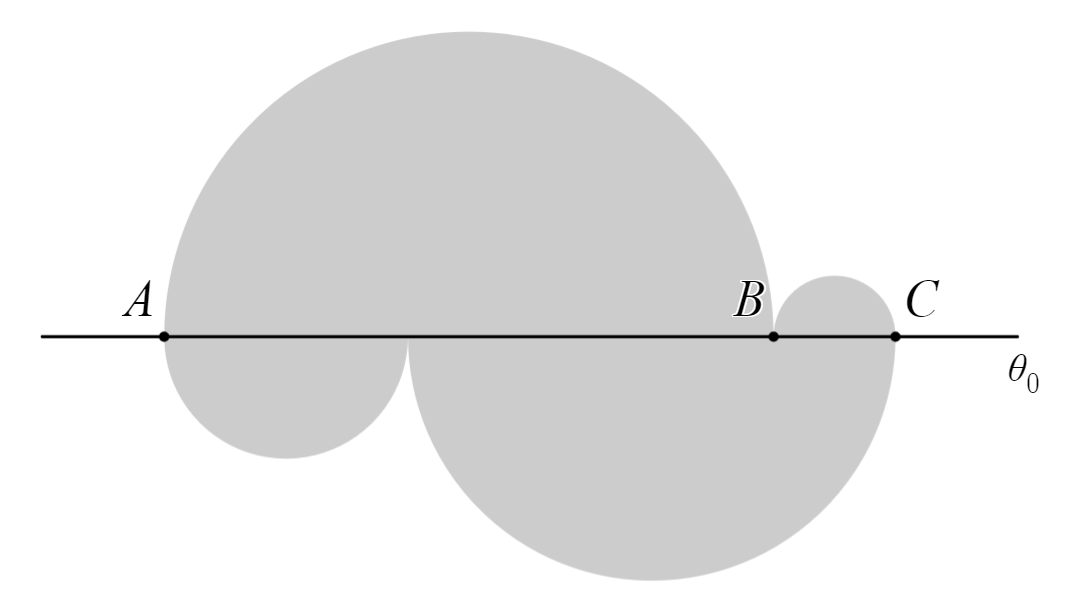}
\caption{Domain for which $H\mapsto\lambda_{H}(\O)$ is discontinuous}\label{figure1}
\end{figure}
\end{example}

However, if we assume convexity of the domain, we get a kind of continuity for both maps.

\begin{teor}\label{layercontinuous} Let $\O$ be a convex domain and $k \in \{0,1,\dots,d\}$. The maps
\[\begin{split}
&H\mapsto T_H(\O)\qquad\text{from $\HH_k\to\R$}\\
&H\mapsto \lambda_H(\O)\qquad\text{from $\HH_k\to\R$}\\
\end{split}\]
are continuous.
\end{teor}

\begin{proof} If $k=0$ then it is obvious and if $k = d$, then it is true because the maps are locally Lipschitz in the set of norms, so assume $k \in \{1,\dots,d-1\}$. Let $H \in \HH_k$, we now prove that both maps are continuous on $H$. Take a sequence $\{H_n\} \subset \HH_k$ converging to $H$. Therefore there are $A_n \in O(d)$ such that $H_{n}(A_n(\xi,\eta)) = G_{n}(\eta)$ for every $(\xi,\eta) \in \R^{d-k}\times\R^k$, where $G_n$ are norms in $\R^k$. By going through subsequences we can assume $A_n$ converges to $A \in O(d)$ and $G_n$ converges to a norm $G$ in $\R^k$ satisfying $H(A(\xi,\eta)) = G(\eta)$ for every $(\xi,\eta) \in \R^{d-k}\times\R^k$.

We first prove the torsion map is continuous. Since $\O$ is convex, then $(A_n^T\O)_x$ converges to $(A^T\O)_x$ in the Hausdorff sense for every $x \in \R^{d-k}$. Therefore, for $n$ large enough we have
\begin{align*}
|T_{G_n}((A_n^T\O)_x)-T_{G}((A^T\O)_x)| &\leq |T_{G_n}((A_n^T\O))_x-T_{G}((A_n^T\O)_x)| + |T_{G}((A_n^T\O)_x)-T_{G}((A^T\O)_x)| \\
& \leq C_{G}T_{G}((A_n^T\O)_x)\lVert G_n - G\rVert + |T_{G}((A_n^T\O)_x)-T_{G}((A^T\O)_x)|,
\end{align*}
which letting $n \to \infty$ implies that $T_{G_n}((A_n^T\O)_x)$ converges to $T_{G}((A^T\O)_x)$ for every $x \in \R^{d-k}$, hence, by dominated convergence theorem
\begin{align*}
\lim_{n \to \infty} T_{H_n}(\O) &= \lim_{n \to \infty} T_{H_n\circ A_n}(A_n^T\O) = \lim_{n \to \infty} \int_{\R^{d-k}}T_{G_n}((A_n^T\O)_x) \ dx \\ 
&= \int_{\R^{d-k}}T_{G}((A^T\O)_x) \ dx = T_{H}(\O).
\end{align*}

Now, we prove the eigenvalue map is continuous. Since there is a compact $K \subset \R^{d-k}$ satisfying
\[
\lambda_{H_n}(\O) = \lambda_{H_{n}\circ A_n}(A_n^T\O) = \inf_{x \in \R^{d-k}}\lambda_{G_n}((A_n^T\O)_x) = \inf_{x \in K}\lambda_{G_n}((A_n^T\O)_x),
\]
there is a sequence $\{x_{k,n}\} \subset K$ such that
\[
\lim_{k \to \infty}\lambda_{G_n}((A_n^T\O)_{x_{k,n}}) = \lambda_{H_n}(\O).
\]
By going through a subsequence we can assume $x_{k,n}$ converges to some $x_n \in K$ and since $\O$ is convex and $\lambda_{G_n}(\cdot)$ is continuous on the Hausdorff complementary topology we have $\lambda_{H_n}(\O) = \lambda_{G_n}((A_n^T\O)_{x_{n}})$. Again, by going through a subsequence we have that $x_{n}$ converges to $x_0$ and since $\O$ is convex, then $(A_n^T\O)_{x_n}$ converges to $(A^T\O)_x$ in the Hausdorff sense. Therefore, for $n$ large enough we have
\begin{align*}
|\lambda_{G_n}((A_n^T\O)_{x_n})-\lambda_{G}((A^T\O)_{x_0})| &\leq |\lambda_{G_n}((A_n^T\O))_{x_n}-\lambda_{G}((A_n^T\O)_{x_n})| + |\lambda_{G}((A_n^T\O)_{x_n})-\lambda_{G}((A^T\O)_{x_0})| \\
& \leq C_{G}\lambda_{G}((A_n^T\O)_{x_n})\lVert G_n - G\rVert + |\lambda_{G}((A_n^T\O)_{x_n})-\lambda_{G}((A^T\O)_{x_0})|,
\end{align*}
which letting $n \to \infty$ implies that $\lambda_{G_n}((A_n^T\O)_{x_n})$ converges to $\lambda_{G}((A^T\O)_{x_0})$. Therefore
\begin{align*}
\lim_{n \to \infty} \lambda_{H_n}(\O) = \lim_{n\to \infty} \lambda_{G_n}((A_n^T\O)_{x_n}) = \lambda_{G}((A^T\O)_{x_0}) \geq \inf_{x \in \R^{d-k}}   \lambda_{G}((A^T\O)_{x}) = \lambda_{H\circ A}(A^T\O) = \lambda_{H}(\O).
\end{align*}
This together with the uppper semicontinuity of $H \mapsto \lambda_{H}(\O)$ concludes the proof.
\end{proof}

In particular, the two-dimensional case provides a stronger continuity, the proof for the map $H\mapsto\lambda_{H}(\O)$ can be seen in \cite{HHM} and for $H\mapsto T_{H}(\O)$ is analogous.

\begin{cor}\label{cor2dim} Let $\O \subset \R^2$ be a convex domain. The maps
\[\begin{split}
&H\mapsto T_H(\O)\qquad\text{from $\HH\to\R$}\\
&H\mapsto \lambda_H(\O)\qquad\text{from $\HH\to\R$}\\
\end{split}\]
are continuous. Also, if $\O \subset \R^2$ is a $C^1$ domain, then the map
$$H\mapsto T_H(\O)\qquad\text{from $\HH\to\R$}$$
is continuous.
\end{cor}

When we restrict the analysis to seminorms induced by quadratic forms, Theorem~\ref{layercontinuous} can be improved, without the need of fixing codimension of $\ker{H}$.

\begin{teor}\label{contquadratic} Let $\O$ be a convex domain. The maps
\[\begin{split}
&H\mapsto T_H(\O)\qquad\text{from $\mathcal{Q}\to\R$}\\
&H\mapsto \lambda_H(\O)\qquad\text{from $\mathcal{Q}\to\R$}\\
\end{split}\]
are continuous. Also, if $\O$ is a $C^1$ domain, then the map
$$H\mapsto T_H(\O)\qquad\text{from $\mathcal{Q}\to\R$}$$
is continuous.
\end{teor}

\begin{proof} Let $H\in\mathcal{Q}\cap\HH_k$, we now prove that the maps are continuous at $H$. Let $\{H_n\} \subset \mathcal{Q}$ be a sequence converging to $H$, then there are $A_{n} \in O(d)$ such that
\[
H_{n}(A_n\xi) = \sqrt{\sum_{i=1}^{d}\alpha_{i,n}^2\xi_i^2}.
\]
Since $O(d)$ is compact, up to a subsequence $A_n$ converges to some $A \in O(d)$, we can also assume that the coefficients $\alpha_{i,n}$ converge to $\alpha_{i}$. We can choose $A_n$ to be such that $\alpha_{i} = 0$ for $i \in \{1,\dots,d-k\}$. Therefore, if we define the sequence $G_{n}$ as
\[
G_{n}(A_n\xi) = \sqrt{\sum_{i=d-k+1}^{d}\alpha_{i,n}^2\xi_i^2},
\]
we have that the sequence $\{G_n\} \subset \HH_k$ converges to $H$. It is also clear that $H_n \geq G_n$. So, if $\O$ is convex, by the semicontinuities of the maps and by Theorem~\ref{layercontinuous} we have
\[
\lambda_{H}(\O) \geq \limsup_{n \to \infty}\lambda_{H_n}(\O) \geq \liminf_{n \to \infty} \lambda_{H_n}(\O) \geq \limsup_{n \to \infty} \lambda_{G_n}(\O) = \lambda_{H}(\O). 
\]
and
\[
T_{H}(\O) \leq \liminf_{n \to \infty} T_{H_n}(\O) \leq \limsup_{n \to \infty} T_{H_n}(\O) \leq \liminf_{n \to \infty} T_{G_n}(\O) = T_{H}(\O),
\]
proving the continuity of both maps. If $\O$ is $C^1$, then the continuity of $H \mapsto T_{H}(\O)$ follows analagously using Theorem~\ref{layercontinuousC1} instead of Theorem~\ref{layercontinuous}.
\end{proof}

As a consequence of Theorems~\ref{layercontinuousC1} and~\ref{layercontinuous}, we can enhance the results in Theorem~\ref{Tmax}.

\begin{cor} If $\O$ is $C^1$ or convex, then there is $H \in \s(\HH_1)$ such that
\[
T_{\max}(\O) = T_{H}(\O).
\]
\end{cor}

\section{On the existence of optimal seminorms}

\begin{defi} Let $\O\subset\R^d$ be a bounded domain and let $q\in\R$. For every seminorm $H$ we define
\[
F_{q,\O}(H) = \lambda_{H}(\O)T_{H}^{q}(\O).
\]
The isoanisotropic constants associated with the functional $F_{q,\O}$ are defined by:
\[\begin{split}
&M_q(\O) = \sup\big\{F_{q,\O}(H)\ :\ \lVert H \rVert = 1\big\},\\
&m_q(\O) = \inf\big\{F_{q,\O}(H)\ :\ \lVert H \rVert = 1\big\}.
\end{split}\]
\end{defi}

Our first results below show that for large $q$ the functional $F_{q,\O}$ has a minimizer which is actually a norm, and the same happens for the maximization problem when $q$ is small.

\begin{teor}\label{min q large}
Let $\O\subset\R^d$ be a bounded domain. Then there is $\bar{q}>0$ such that for every $q \geq \bar{q}$, we have
\[
m_q(\O) = \inf_{H \in \s(\HH_{d})}F_{q,\O}(H).
\]
In addition, there is $\tilde{q}$ such that for every $q \geq \tilde{q}$, there is $H_{q} \in \s(\HH_{d})$ such that
\[
m_{q}(\O) = F_{q,\O}(H_q).
\]
\end{teor}

\begin{proof} We first prove that
\[
c =\inf_{H \in \s(\HH_{\leq d-1})} \frac{T_{H}(\O)}{T(\O)} > 1.
\]
The proof follows by contradiction, it is clear that $c \geq 1$, so assume $c =1$. Therefore, there is a sequence $\{H_n\}\subset \s(\HH_{\leq d-1})$ such that
\[
\lim_{n\to\infty}T_{H_n}(\O) = T(\O).
\]
Since $\s(\HH_{\leq d-1})$ is compact, the sequence converges to some $H \in \s(\HH_{\leq d-1})$ up to a subsequence. Due to lower semicontinuity of the torsion, we have
\[
T(\O) = \liminf_{n\to\infty}T_{H_n}(\O) \geq T_{H}(\O) > T(\O).
\]
arriving at the contradiction. Therefore if $q > \tilde{q} = \left(\log{(\lambda(\O))}-\log{(\lambda_{\min}(\O))}\right)\left(\log{c}\right)^{-1}$ we have
\[
\frac{\lambda_{H}(\O)T_{H}^q(\O)}{\lambda(\O)T^q(\O)} \geq \left(\frac{\lambda_{\min}(\O)}{\lambda(\O)}\right)c^q > 1,
\]
proving the infimum is can be taken only in $\s(\HH_{d})$. Notice that if $\mathcal{E}$ is the euclidean norm, then we proved that
\[
m_{q}(\O) \leq F_{q,\O}(\mathcal{E})
\]
for every $q \leq \tilde{q}$, hence the infimum can be taken in $F_{q,\O}^{-1}((-\infty,F_{q,\O}(\mathcal{E})])\cap \s(\HH)$. Due to Theorem~\ref{semicont}, the map $H \mapsto F_{q,\O}(H)$ being continuous in $\HH_{d}$, also, as a consequence of the proof of the first statement, we have $F_{q,\O}^{-1}((-\infty,F_{q,\mathcal{E}}(\O)]) \subset \HH_d$. This implies that $F_{q,\O}^{-1}((-\infty,F_{q,\O}(\mathcal{E})])\cap \s(\HH)$ is compact, therefore continuity assures the existence of a minimizer.
\end{proof}

In the same way we proved Theorem~\ref{min q large}, we can prove the following result.

\begin{teor} \label{max q small}
Let $\O\subset\R^d$ be a bounded domain. Then there is $\bar{q} > 0$ such that for every $q \leq\bar{q}$, we have
\[
M_q(\O) = \sup_{H \in \s(\HH_{d})}F_{q,\O}(H).
\]
In addition, there is $\tilde{q} > 0$ such that for every $q \leq \tilde{q}$, there is $H_{q} \in \s(\HH_{d})$ such that
\[
M_{q}(\O) = F_{q,\O}(H_q).
\]
\end{teor} 

We consider now, the particular case of $H \in \mathcal{Q}$ to deal with $q$ small for the infimum problem and large for the supremum problem.

\begin{defi} The quadratic isoanisotropic constants are defined as
\[\begin{split}
&\tilde{M}_q(\O)=\sup\left\{F_{q,\O}(H)\ :\ H\in\s(\mathcal{Q})\right\},\\
&\tilde{m}_q(\O)=\inf\left\{F_{q,\O}(H)\ :\ H\in\s(\mathcal{Q})\right\}.
\end{split}\]
\end{defi}

\begin{propo}\label{euclidean ellipsoid}Let $E$ be an ellipsoid. For every $q \leq 1$ and $j \in \{1,\dots,d\}$, we have
\[
\lambda(E)T^{q}(E) > \lambda_{H_j}(E)T_{H_j}^{q}(E)\left(a_j^2\sum_{i=1}^{d} \frac{1}{a_i^2}\right)^{1-q} > \lambda_{H_j}(E)T_{H_j}^{q}(E).
\]
where $a_i$ are the principal axis of $E$ and $H_j(\xi) = |\xi_j|$.
    
\end{propo}

\begin{proof} Without loss of generality we can assume that $E = E_d(a)$. By subadditivity of $H \mapsto \lambda_{H}(\O)$ we have for every $j \in \{1,\dots,d\}$
\[
\lambda(E) > \sum_{i=1}^{d}\lambda_{H_i}(E) = \sum_{i=1}^{d} \frac{\pi^2}{4a_i^2} = \lambda_{H_j}(E) a_j^2\sum_{i=1}^{d} \frac{1}{a_i^2}.
\]
Due to Corollary~\ref{corellipsoid} we have
\[
T(E) = \frac{\omega_d}{d+2}\left(\prod_{i=1}^{d}a_i\right)\left(\sum_{i=1}^{d} \frac{1}{a_i^2}\right)^{-1} = T_{H_j}(E)\left(a_j^2\sum_{i=1}^{d} \frac{1}{a_i^2}\right)^{-1}
\]
for every $j \in \{1,\dots,d\}$. Putting both inequalities together gives the desired result.
\end{proof}

\begin{teor}\label{teorquadratic} Let $E$ be an ellipsoid with $a_i$ the lengths of its semi-major axes. For every $q \leq 1$, we have
\[
\tilde{m}_q(E) = \inf_{H \in \s(\HH_{1})}F_{q,E}(H) = \frac{\pi^2|E|^q}{4(d+2)^q}\left(\max_{i\in\{1,\dots,d\}}a_i\right)^{2(q-1)}.
\]
Moreover, if $q < 1$ and $e$ is a direction corresponding to the longest semi-major axis, and $H(\xi) = |\langle e,\xi\rangle|$, then
\[
\tilde{m}_{q}(E) = F_{q,E}(H)
\]
and these are the only minimizers. If $q =1$, then every $H \in \s(\HH_1)$ is a minimizer.
\end{teor}

\begin{proof} Take $H \in \mathcal{Q}\cap\HH_d$ with $\lVert H\rVert = 1$, there is $R \in O(d)$ such that
\[
H(R\xi) = \sqrt{\sum_{i=1}^{d}\alpha_i^2\xi_i^2}.
\]
Denote
\[
M =\begin{bmatrix} 
    \alpha_1^{-1} & 0 & \dots \\
    \vdots & \ddots & \\
    0 &        & \alpha_{d}^{-1}.
    \end{bmatrix}
\]
Let $j \in \{1,\dots,d\}$ be such that $\alpha_j = 1$. From Proposition~\ref{propochange} and~\ref{euclidean ellipsoid}, denoting the principal axis of $MR^TE$ as $a_{i,H}$, we have
\begin{align*}
F_{q,E}(H) &= \lambda_{H}(E)T_{H}^{q}(E) = \lambda_{H\circ R}(R^TE)T_{H \circ R}^q(R^T E) = \lambda(MR^TE)|\det M|^{-q}T^{q}(MR^TE) \\
&> \left(a_{j,H}^2\sum_{i=1}^{d} \frac{1}{a_{i,H}^2}\right)^{1-q}\lambda_{H_j}(MR^TE)|\det M|^{-q}T_{H_j}^{q}(MR^TE) \\ 
&= \left(a_{j,H}^2\sum_{i=1}^{d} \frac{1}{a_{i,H}^2}\right)^{1-q}\lambda_{H_j \circ M^{-1}}(R^TE) T_{H_j \circ M^{-1}}(R^TE) \\
&= \left(a_{j,H}^2\sum_{i=1}^{d} \frac{1}{a_{i,H}^2}\right)^{1-q}\lambda_{H_j \circ M^{-1}\circ R^T}(E)T_{H_j \circ M^{-1}\circ R^T}^{q}(E) \geq \tilde{m}_{q}(E),
\end{align*}
where the last inequality comes from the fact that
$\lVert H_j \circ M^{-1}\circ R^T\rVert = 1$, since
\[
(H_j \circ M^{-1}\circ R^T)(\xi) = |\langle M^{-1}R^T\xi,e_{j}\rangle| = |\langle R^{T}\xi, M^{-1}e_j\rangle| = |\langle \xi, Re_j\rangle|.
\]
Now, let $H \in \s(\HH_{k})\cap\mathcal{Q}$ with $k \in \{2,3,\dots,d-1\}$. There is a sequence $\{G_n\} \subset \s(\HH_d)\cap\mathcal{Q}$ that converges to $H$. Let $R_n$ be such that
\[
G_n(R_n\xi) = \sqrt{\sum_{i=1}^{d}\alpha_{i,n}^2\xi_i^2},
\]
where $\alpha_{d,n} = 1$ for every $n$ and
\[
\lim \alpha_{i,n} = \alpha_{i} \ \text{for every} \ i \in \{1,\dots,d\} , \text{where} \ \alpha_{i} = 0 \ \text{for every} \ i \in \{1,\dots,d-k\}.
\]
We can assume that up to a subsequence $R_n$ converges to some $R$ satisfying
\[
H(R\xi) = \sqrt{\sum_{i=d-k+1}^{d}\alpha_{i}^2\xi_i^2}
\]
Denote
\[
M_n =\begin{bmatrix} 
    \alpha_{1,n}^{-1} & 0 & \dots \\
    \vdots & \ddots & \\
    0 &        & \alpha_{d,n}^{-1}
    \end{bmatrix},
\]
by our last computation, we have
\begin{equation}\label{ineqGn}
F_{q,E}(G_n) > \left(a_{d,G_n}^2\sum_{i=1}^{d} \frac{1}{a_{i,G_n}^2}\right)^{1-q} \tilde{m}_{q}(E).
\end{equation}
Notice that the choice of which axis of $M_nR_n^TE$ is $a_{i,G_n}$ is free, so we choose it such that $a_{i,G_n} \geq a_{j,G_n}$ if $i \leq j$, so there is $A_n \in O(d)$ such that
\[
A_nM_nR_n^TE = E_{d}((a_{1,G_n},\dots,a_{d,G_n})) =: E_n.
\]
We aim to prove that at least $k$ of these axis have a positive limit. Notice first that the limit of the axis are never zero due to \eqref{ineqGn} and $F_{q,E}(G_n)$ being bounded. Assume by contradiction that
\[
\lim_{n \to \infty}a_{i,G_n} = +\infty \ \text{for every} \ i \in\{1,\dots,d-k+1\}
\]
therefore, in the Hausdorff sense we have
\[
\lim E_n = \R^{d-k+1}\times E',
\]
where $E'$ is an ellipsoid in $\R^{k-1}$, which means that for every linear subspace $V$ of dimension $k$ we have
\[
\mathcal{H}^{k}(V\cap\lim E_n) = +\infty.
\]
On the other hand, take $W_n = A_n(\{0\}\times \R^k)$ and $W = \lim W_n$. Notice that since $\lim \alpha_{i,n} > 0$ for $i \in \{d-k+1,\dots,k\}$ then there is a compact set $D \subset \R^k$ such that $M_nR_n^TE\cap (\{0\}\times \R^k) \subset \{0\}\times D$, therefore
\begin{align*}
\mathcal{H}^{k}(\lim E_n \cap W) &= \lim_{n \to \infty}\mathcal{H}^{k}(A_nM_nR_n^TE\cap A_nW_n) = \lim_{n \to \infty} \mathcal{H}^k(A_n(M_nR_n^TE\cap (\{0\}\times \R^k))) \\
&\leq \limsup_{n \to \infty} \mathcal{H}^{k}(A_n(\{0\}\times D)) < +\infty,
\end{align*}
giving a contradiction, hence there is $i_0 \in \{1,\dots,d-k+1\}$ such that $\lim a_{i_0,G_n} < +\infty$, which implies that $\gamma_{i} : =\lim a_{i,G_n} < +\infty$ for $i \in \{d-k+1,\dots,d\}$. Therefore, using this and Theorem~\ref{teorquadratic} in \eqref{ineqGn} we have
\begin{align*}
F_{q,E}(H) &= \lim_{n \to \infty}F_{q,E}(G_n) \geq \lim_{n\to \infty}\left(a_{d,G_n}^2\sum_{i=1}^{d} \frac{1}{a_{i,G_n}^2}\right)^{1-q} \tilde{m}_{q}(E) \\
&\geq \lim_{n\to \infty}\left(a_{d,G_n}^2\sum_{i=d-k+1}^{d} \frac{1}{a_{i,G_n}^2}\right)^{1-q} \tilde{m}_{q}(E) = \left(\gamma_{d}^2\sum_{i=d-k+1}^{d} \frac{1}{\gamma_{i}^2}\right)^{1-q} \tilde{m}_{q}(E) > \tilde{m}_{q}(E).
\end{align*}
We now compute $\tilde{m}_{q}(E)$, for each $v \in \s^{d-1}$ we define $H_v(\xi) = |\langle v,\xi\rangle|$. Let $A \in O(d)$ be such that $AE = E_d(a)$. We can assume without loss of generality $a_1 \leq a_2 \leq \dots \leq a_d$, write $Av =\tilde{v} = (\tilde{v}_1,\dots,\tilde{v}_d)$, from Proposition~\ref{propoellipsoid}, we have
\begin{align*}
\lambda_{H_v}(E)T_{H_v}^q(E) &= \lambda_{H_{\tilde{v}}}(AE)T_{H_{\tilde{v}}}^q(AE) =\frac{\pi^2|E|^q}{4(d+2)^q}\left(\sum_{i=1}^{d}\frac{\tilde{v_i}^2}{a_i^2}\right)^{1-q} \\ 
&= \frac{\pi^2|E|^q}{4(d+2)^q}\left(\sum_{i=1}^{d-1}\frac{\tilde{v_i}^2}{a_i^2}+\left(1-\sum_{i=1}^{d-1}\tilde{v_i}^2\right)\frac{1}{a_d^2}\right)^{1-q} \\
&=\frac{\pi^2|E|^q}{4(d+2)^q}\left(\sum_{i=1}^{d-1}\left(\frac{1}{a_i^2}-\frac{1}{a_d^2}\right)\tilde{v_i}^2+\frac{1}{a_d^2}\right)^{1-q} \geq \frac{\pi^2|E|^q}{4(d+2)^q}\left(\frac{1}{a_d^2}\right)^{1-q}.
\end{align*}
Notice that this inequality is actually an equality if we take $v = A^Te_d$, hence
\[
\tilde{m}_{q}(E) = \frac{\pi^2|E|^q}{4(d+2)^q}a_d^{2(q-1)},
\]
as required.
\end{proof}

As a consequence of Theorem~\ref{teorquadratic} together with Theorem~\ref{contquadratic} we have

\begin{cor}\label{cormaxquadratic} Let $E$ be an ellipsoid. For every $q \leq 1$, there is $H \in \s(\HH)\setminus\s(\HH_1)$ such that
\[
\tilde{M}_q(E) = F_{q,E}(H).
\]    
\end{cor}

It is important to notice that it is not clear that the threshold for $\tilde{m}_{q}(E)$ is $q = 1$. Due to Theorem~\ref{contquadratic} there is always a minimizer for $\tilde{m}_{q}(E)$, however, for $q > 1$ we do not have a characterization for it as in Theorem~\ref{teorquadratic}. It is expected that in this case, the minimizers will no longer be in $\s(\HH_1)$, and we give a result in this direction.

\begin{propo}\label{ell2}
Let $E$ be an ellipsoid with $a_i$ the lengths of its semi-major axes ordered as $a_1\ge a_2\ge\dots\ge a_d$ and let
\[
q_E = 1 + \frac{\log{2}}{\log{\left(1+\frac{a_{d}^2}{a_{d-1}^2}\right)}}.
\]
If $q > q_E$, then
\[
m_q(E) = \inf_{H \in \s(\HH)\setminus\s(\HH_{1})}F_{q,E}(H).
\]
and
\[
\tilde{m}_q(E) = \inf_{H \in \s(\mathcal{Q})\setminus\s(\HH_{1})}F_{q,E}(H).
\]
\end{propo}

\begin{proof} We can assume without loss of generality that $E = E_d(a)$ with $a = (a_1,a_2,\dots,a_d)$. Denote for $\alpha \in [0,1]$ the seminorm
\[
H_{\alpha}(\xi) = \sqrt{\alpha^2\xi_{d-1}^2+\xi_d^2}.
\]
Notice that $H_{\alpha} \in \s(\HH)$ for every $\alpha \in [0,1]$. Assume now $q > q_E > 1$, from the computations of the proof of Theorem~\ref{teorquadratic} we have
\[
\inf_{H \in \s(\HH_{1})}F_{q,E}(H) = \frac{\pi^2|E|^q}{4(d+2)^q}a_d^{2(q-1)} = F_{q,E}(H_0).
\]
Take $\alpha \in (0,1]$ such that
\[
\alpha > \frac{a_{d-1}^2}{a_d^2}\left(2^{\frac{1}{q-1}}-1\right),
\]
this is possible because $q > q_E$. We compute $T_{H_{\alpha}}(E)$ using Proposition~\ref{propocomputedegenerate} and Proposition~\ref{propochange}. Let $G_{\alpha}(x,y) = \sqrt{\alpha^2x^2+y^2}$ be a norm in $\R^2$, then
\begin{align*}
T_{H_{\alpha}}(E) &= \int_{E_{d-2}(a_1.\dots,a_{d-2})} T_{G_{\alpha}}\left(\sqrt{1-\sum_{i=1}^{d-2}\frac{x_i^2}{a_i^2}}E_{2}(a_{d-1},a_d)\right) \ dx \\
&= \int_{E_{d-2}(a_1.\dots,a_{d-2})} \left(1-\sum_{i=1}^{d-2}\frac{x_i^2}{a_i^2}\right)^{2}T_{G_{\alpha}}\left(E_{2}(a_{d-1},a_d)\right) \ dx \\
&= T_{G_{\alpha}}\left(E_{2}(a_{d-1},a_d)\right)\left(\prod_{i=1}^{d-2}a_i\right)\int_{B^{d-2}}(1-|x|^2)^{2} \ dx \\
&= \alpha T(E_{2}(\alpha^{-1}a_{d-1},a_d))\left(\prod_{i=1}^{d-2}a_i\right)\frac{4\omega_d}{(d+2)\pi} \\
&= \alpha\frac{\pi\alpha^{-1}a_{d-1}a_d}{4}\left(\frac{1}{a_d^2}+\frac{\alpha^2}{a_{d-1}^2}\right)^{-1}\left(\prod_{i=1}^{d-2}a_i\right)\frac{4\omega_d}{(d+2)\pi} \\
&=\frac{\omega_d}{d+2}\left(\prod_{i=1}^{d}a_i\right)\left(\frac{1}{a_d^2}+\frac{\alpha^2}{a_{d-1}^2}\right)^{-1}.
\end{align*}
Also, from Proposition~\ref{propocomputedegenerate} and Proposition~\ref{propochange} we have
\[
\lambda_{H_{\alpha}}(E_{d}(a)) = \lambda_{G_{\alpha}}(E_{2}(a_{d-1},a_d)) = \lambda(E_{2}(\alpha^{-1}a_{d-1},a_d)).
\]
Now, notice that the following rectangle is contained in $E_{2}(\alpha^{-1}a_{d-1},a_d)$
\[
R_{\alpha} := \left(\frac{\alpha}{a_{d-1}}+\frac{1}{a_d}\right)^{-\frac{1}{2}}\left(-\sqrt{\frac{a_{d-1}}{\alpha}},\sqrt{\frac{a_{d-1}}{\alpha}}\right)\times\left(-\sqrt{a_d},\sqrt{a_d}\right) \subset E_{2}(\alpha^{-1}a_{d-1},a_d).
\]
Therefore
\begin{align*}
\frac{F_{q,E}(H_{\alpha})}{F_{q,E}(H_0)} &= \lambda(E_{2}(\alpha^{-1}a_{d-1},a_d))\left(\frac{\omega_d}{d+2}\left(\prod_{i=1}^{d}a_i\right)\left(\frac{1}{a_d^2}+\frac{\alpha^2}{a_{d-1}^2}\right)^{-1}\right)^q\left(\frac{\pi^2}{4a_d^2}\left(\frac{\omega_d}{d+2}\left(\prod_{i=1}^{d}a_i\right)a_d^2\right)^q\right)^{-1} \\
&= \frac{4a_d^2\lambda(E_{2}(\alpha^{-1}a_{d-1},a_d))}{\pi^2}\left(a_d^2\left(\frac{1}{a_d^2}+\frac{\alpha^2}{a_{d-1}^2}\right)\right)^{-q} \\
&\leq \frac{4a_d^2\lambda(R_{\alpha})}{\pi^2}\left(a_d^2\left(\frac{1}{a_d^2}+\frac{\alpha^2}{a_{d-1}^2}\right)\right)^{-q} = a_d^2\left(\frac{1}{a_d}+\frac{\alpha}{a_{d-2}}\right)^2\left(a_d^2\left(\frac{1}{a_d^2}+\frac{\alpha^2}{a_{d-1}^2}\right)\right)^{-q} \\
&\leq 2\left(a_d^2\left(\frac{1}{a_d^2}+\frac{\alpha^2}{a_{d-1}^2}\right)\right)^{1-q} < 1,
\end{align*}
where the last line is due to $q > q_E > 1$. Hence
\[
m_{q}(E) \leq \tilde{m}_{q}(E) \leq F_{q,E}(H_{\alpha}) < F_{q,E}(H_0) = \inf_{H \in \s(\HH_{1})}F_{q,E}(H),
\]
which concludes the proof.
\end{proof}

Regarding the study of $\tilde{M}_{q}(\O)$, we give a result when $\O = B^d$, but first a technical lemma.

\begin{lema}\label{lemaboundball} Let $\alpha_i\ge0$ and let
\[
H(\xi) = \sqrt{\sum_{i=1}^{d}\alpha_i^2\xi_i^2}.
\]
Then
\[
\lambda_{H}(B^d) \leq \left(\frac{1}{d}\sum_{i=1}^{d}\alpha_i^2\right)\lambda(B^d).
\]
\end{lema}

\begin{proof} Let $u$ be the eigenfunction associated to $\lambda(B^d)$ normalized by $\lVert u\rVert_2=1$. Since $u$ is radial we can write $u(x) = f(|x|)$ for some function $f \colon [0,\infty) \to \R$. Therefore
\begin{align*}
\lambda_{H}(B^d) &\leq \int_{B^d}H^2(\nabla u)\,dx = \int_{B^d} H^2\left(f'(|x|)\frac{x}{|x|}\right)\,dx\\
&= \int_{0}^1\int_{\s^{d-1}}(f'(r))^2 H^2(\omega)r^{d-1}\,d\omega\,dr\\
&= \left(\int_{0}^{1}(f'(r))^2 r^{d-1}\,dr\right) \int_{\s^{d-1}}\sum_{i=1}^{d}\alpha_i^2\omega_i^2\,d\omega\\
&= \left(\frac{1}{d\omega_d}\int_{0}^1\int_{\s^{d-1}}(f'(r))^2 r^{d-1}\,d\omega\,dr\right)\left(\sum_{i=1}^{d}\alpha_i^2\int_{\s^{d-1}}\omega_i^2\,d\omega\right) \\
&= \frac{\lambda(B^d)}{d\omega_d}\omega_d\sum_{i=1}^{d}\alpha_i^2 = \left(\frac{1}{d}\sum_{i=1}^{d}\alpha_i^2\right)\lambda(B^d),
\end{align*}
as required.
\end{proof}

\begin{teor}\label{teorquadratic2} For every $q \leq 1$, we have
\[
\tilde{M}_q(B^d) = \lambda(B^d)T^{q}(B^d)
\]
Moreover, the Euclidean norm is the only maximizer for the functional $F_{q,B^d}$.
\end{teor}

\begin{proof} Let $H \in \s(\mathcal{Q})$, without loss of generality we can assume that there are $\alpha_i \geq 0$, with $\max \alpha_i =1$, satisfying
\[
H(\xi) = \sqrt{\sum_{i=1}^{d}\alpha_i^2\xi_i^2}.
\]
Assume first that $\alpha_i > 0$ for every $i \in \{1,\dots,d\}$, by Proposition~\ref{propochange}
\begin{align*}
T_{H}(B^d)&=\left(\prod_{i=1}^{d} \alpha_i^{-1}\right)^{-1} T(E_d(\alpha_1^{-1},\dots,\alpha_{d}^{-1}))\\
&=\left(\prod_{i=1}^{d}\alpha_i\right)\frac{\omega_d}{(d+2)}\left(\prod_{i=1}^{d}\alpha_i^{-1}\right)\left(\sum_{i=1}^{d} \alpha_i^2\right)^{-1}\\
&=\frac{\omega_d}{(d+2)}\left(\sum_{i=1}^{d} \alpha_i^2\right)^{-1}.
\end{align*}
By a continuity argument, using Theorem~\ref{contquadratic} we have
\begin{equation}\label{eqTHBd}
T_{H}(B^d) = \frac{\omega_d}{(d+2)}\left(\sum_{i=1}^{d} \alpha_i^2\right)^{-1}
\end{equation}
for $\alpha_i \geq 0$. By Lemma~\ref{lemaboundball} and \eqref{eqTHBd} we have
\begin{align*}
F_{q,B^d}(H)&=\lambda_{H}(B^d)T^{q}_{H}(B^d)\\
&\le\left(\frac{1}{d}\sum_{i=1}^{d}\alpha_i^2\right)\lambda(B^d)\left(\frac{\omega_d}{(d+2)}\left(\sum_{i=1}^{d} \alpha_i^2\right)^{-1}\right)^q\\
&=\left(\frac{1}{d}\sum_{i=1}^{d}\alpha_i^2\right)\lambda(B^d) \left(dT(B^d)\left(\sum_{i=1}^{d} \alpha_i^2\right)^{-1}\right)^{q}\\
&=\left(\frac{1}{d}\sum_{i=1}^{d}\alpha_i^2\right)^{1-q} \lambda(B^d)T^{q}(B^d).
\end{align*}
Notice that if $H$ is not the Euclidean norm, then there is $i_0$ such that $a_{i_0} < 1$, hence
\[
F_{q,B^d}(H) < \lambda(B^d)T^{q}(B^d),
\]
which concludes the proof.
\end{proof}

\begin{rem}
Putting together Theorem \ref{teorquadratic} and Proposition \ref{ell2}, in dimension $d=2$ and for the ellipse
$$E=\left\{\frac{x^2}{a^2}+\frac{y^2}{b^2}<1\right\},$$
with $a\le b$, we obtain:
\begin{itemize}
\item for every $q\le1$
$$\tilde{m}_q(E)=\frac{\pi^2|E|^q}{4^{q+1}}b^{2(q-1)},$$
achieved for $H(\xi)=|\xi_2|$; 
\item for every
$$q\ge1+\frac{\log 2}{\log{\left(1+(a/b)^2\right)}}$$
$\tilde{m}_q(E)$ and $m_{q}(E)$ are achieved on a norm. 
\end{itemize}
In particular, if $E$ is the unit disc $D$, for every $q\le1$ the minimal value $\tilde{m}_q(D)$ is achieved on any $H(\xi)=|\langle e,\xi\rangle|$ with $|e|=1$, and for every $q\ge2$, $\tilde{m}_q(D)$ and $m_{q}(D)$ are achieved on a norm.
\end{rem}

We summarize in Tables \ref{tableAny}--\ref{tableC1} what is known on the problems $m_{q}(\O), M_{q}(\O), \tilde{m}_{q}(E)$ and $\tilde{M}_{q}(E)$. These are the results of Theorem~\ref{min q large},~\ref{max q small},~\ref{teorquadratic},~\ref{teorquadratic2} Corollary~\ref{cormaxquadratic} and consequences of Theorem~\ref{contquadratic} and Corollary~\ref{cor2dim}.

\begin{table}[h!]
\begin{tabular}{|c|c|c|}
\hline
q & $m_{q}(\O)$ & $M_{q}(\O)$ \\
\hline
small & -  & Achieved on $\s(\HH_d)$ \\ 
\hline
large & Achieved on $\s(\HH_d)$ & - \\ 
\hline
\end{tabular}
\caption{$\O$ is any bounded domain}
\label{tableAny}
\end{table}

\begin{table}[h!]
\begin{tabular}{|c|c|c|}
\hline
q & $\tilde{m}_{q}(E)$ & $\tilde{M}_{q}(E)$ \\
\hline
small & Achieved on $\s(\HH_1)$ and characterized  & Achieved on $\s(\HH_d)$ \\ 
\hline
$q \leq 1$ & Achieved on $\s(\HH_1)$ and characterized & Achieved on $\s(\HH)\setminus\s(\HH_1)$ (or by $|\cdot|$ if $E = B^d$)\\ 
\hline
$q > 1$& Achieved on $\s(\HH)$ & Achieved on $\s(\HH)$ \\
\hline
large & Achieved on $\s(\HH_d)$ & Achieved on $\s(\HH)$\\
\hline
\end{tabular}
\caption{$E$ is an ellipsoid}
\label{tableQuadratic}
\end{table}

\begin{table}[h!]
\begin{tabular}{|c|c|c|}
\hline
q & $m_{q}(\O)$ & $M_{q}(\O)$ \\
\hline
small & Achieved on $\s(\HH)$  & Achieved on $\s(\HH_d)$ \\ 
\hline
large & Achieved on $\s(\HH_d)$ & Achieved on $\s(\HH)$ \\ 
\hline
\end{tabular}
\caption{$\O \subset \R^2$ is convex}
\label{tableConvex}
\end{table}

\begin{table}[h!]
\begin{tabular}{|c|c|c|}
\hline
q & $m_{q}(\O)$ & $M_{q}(\O)$ \\
\hline
small & -  & Achieved on $\s(\HH_d)$ \\ 
\hline
large & Achieved on $\s(\HH_d)$ & Achieved on $\s(\HH)$ \\ 
\hline
\end{tabular}
\caption{$\O \subset \R^2$ is $C^1$}
\label{tableC1}
\end{table}

\section{Remarks on some inequalities}

We recall that, when $H(\xi)=|\xi|$ the following Kohler-Jobin inequality holds:
\[\lambda(\O)T^q(\O)\ge\lambda(B)T^q(B)\]
for every bounded domain $\O$ with $|\O|=1$, whenever $q\le d/(d+2)$. Theorem~\ref{degeneratekohlerjobin} shows that a similar inequality cannot be expected in the case of degenerate seminorms $H$. We first give a proof of important bounds on the functional $F_{1,\O}(H)$.

\begin{propo}\label{propoupperboundq=1} Let $\O\subset\R^d$ be bounded domain and let $H \in \HH$. Then,
\[
\lambda_{H}(\O)T_{H}(\O) \leq |\O|.
\]
In addition, if $H \in \HH_{k}$ and $\O$ is convex, then
\[
\lambda_{H}(\O)T_{H}(\O) \geq \frac{\pi^2|\O|}{4kd^{s(d+2)}(d+2)},
\]
where $s = \frac{1}{2}$ if $\O$ is symmetric and $s=1$ if not. If we also have $k=1$, then
\[
\lambda_{H}(\O)T_{H}(\O) \leq \frac{\pi^2}{12}|\O|.
\]
\end{propo}

\begin{proof} For every $\varepsilon > 0$, there is a function $w_{\varepsilon}$ such that
\[
T_{H}(\O) \leq (1+\varepsilon)\left(\int_{\O} w_{\varepsilon} \ dx\right)^{2}\left(\int_{\O} H^{2}(\nabla w_{\varepsilon}) \ dx\right)^{-1}.
\]
Therefore, by H\"older inequality we have
\begin{align*}
\lambda_{H}(\O)T_{H}(\O) &\leq \frac{\left(\int_{\O} H^{2}(\nabla w_{\varepsilon}) \ dx\right)}{\int_{\O} w_{\varepsilon}^2}\frac{(1+\varepsilon)\left(\int_{\O} w_{\varepsilon} \ dx\right)^{2}}{\left(\int_{\O} H^{2}(\nabla w_{\varepsilon}) \ dx\right)} \leq (1+\varepsilon)|\O|,
\end{align*}
the result now follows by letting $\varepsilon$ go to zero. 

Now assume that $\O$ is convex and that $H \in \HH_{k}$, up to a rotation we can assume that $H(\xi,\eta) = G(\eta)$ for every $(\xi,\eta) \in \R^{d-k}\times\R^k$ for a norm $G$ in $\R^k$. Due to Theorem III in \cite{J}, there is a norm $N \in \mathcal{Q}$ in $\R^k$ such that $\sqrt{k}N \geq G \geq N$, extending this to a seminorm $\tilde{N}(\xi,\eta) = N(\eta)$ we have $\sqrt{k}\tilde{N} \geq H \geq \tilde{N}$. Also from Theorem III in \cite{J}, there is an ellipsoid $E$ such that $E \subseteq \O \subseteq d^{s}E$, where $s = \frac{1}{2}$ if $\O$ is symmetric and $s =1$ if not. Notice that $|E| \geq d^{-sd}|\O|$. Therefore, from Theorem~\ref{teorquadratic} we have
\begin{align*}
\lambda_{H}(\O)T_{H}(\O) \geq \lambda_{\tilde{N}}(d^sE)T_{\sqrt{k}\tilde{N}}(E) = \frac{1}{kd^{2s}}\lambda_{\tilde{N}}(E)T_{\tilde{N}}(E) \geq \frac{\pi^2|E|}{4kd^{2s}(d+2)} \geq \frac{\pi^2|\O|}{4kd^{s(d+2)}(d+2)}.
\end{align*}

Now assume that $H \in \HH_1$. Hence, up to a rotation $H(\xi) = |\xi_d|$. Notice also that $\O_x$ is always an interval, therefore
\begin{align*}
\lambda_{H}(\O)T_{H}(\O) & = \frac{\pi^2}{(L_{e_d}(\O))^2}\int_{\R^{d-1}}\frac{|\O_x|^3}{12} \ dx \\
&= \frac{\pi^2}{12}\int_{\R^{d-1}}\frac{|\O_x|^2}{(L_{e_d}(\O))^2} |\O_x| \ dx \leq \frac{\pi^2}{12}\int_{\R^{d-1}}|\O_x| \ dx  = \frac{\pi^2|\O|}{12},
\end{align*}
which concludes the proof.
\end{proof}

\begin{teor}[Degenerate Kohler-Jobin]\label{degeneratekohlerjobin} Let $H \in \HH_{\leq d-1}$. Then, for every $q \in \R$,
\[
\inf_{|\O| = 1} \lambda_{H}(\O)T_{H}^{q}(\O) = 0.
\]
\end{teor}

\begin{proof} 
Let $k \in \{1,\dots,d-1\}$ be such that $H \in \HH_{k}$, without loss of generality we can assume $H(\xi,\eta) = G(\eta)$ for every $(\xi,\eta)\in\R^{d-k}\times\R^{k}$, where $G$ is a norm in $\R^k$. There are $c_1,c_2 > 0$ such that $c_2|\eta| \geq G(\eta) \geq c_1|\eta|$, hence, if we denote
\[
\mathcal{E}_{k}(\xi) = \sqrt{\sum_{i=d-k+1}^{d}\xi_i^2},
\]
we have $c_2\mathcal{E}_k \geq G \geq c_1\mathcal{E}_k$ and therefore
\[
\lambda_{H}(\O)T_{H}^q(\O) \leq c_2^2c_1^{-2q}\lambda_{\mathcal{E}_k}(\O)T_{\mathcal{E}_k}^q(\O)
\]
for every bounded domain $\O$, hence it is enough to prove the theorem for $\mathcal{E}_k$. 

First we assume $q < 1$, consider $\O_n = (0,1)^{d-k-1}\times\left(0,\omega_k^{-1}n^{-k}\right)\times (nB^k)$, then $|\O_n| = 1$. Due to Proposition~\ref{propocomputedegenerate} we have
\[
\lambda_{\mathcal{E}_k}(\O_n) = \inf_{x \in \R^{d-k}}\lambda_{\mathcal{E}_k}((\O_n)_x) = \inf_{x \in (0,1)^{d-k-1}\times\left(0,\omega_k^{-1}n^{-k}\right)} \lambda(nB^k) = n^{-2}\lambda(B^k)
\]
and
\[
T_{\mathcal{E}_k}(\O_n) = \int_{(0,1)^{d-k-1}\times\left(0,\omega_k^{-1}n^{-k}\right)}T(nB^k) \ dx = \omega_k^{-1}n^{-k}n^{k+2}T(B^k) = \frac{n^2}{k(k+2)}.
\]
Therefore
\[
\lim_{n \to \infty} \lambda_{\mathcal{E}_k}(\O_n)T_{\mathcal{E}_k}^q(\O_n) = \lim_{n \to \infty} n^{-2}\lambda(B^k)\frac{n^{2q}}{k^q(k+2)^q} = 0.
\]

Now assume $q \geq 1 > \frac{2}{k+2}$, from Proposition 2.1 of \cite{BBP22} there is a sequence of bounded domains $U_n \subset \R^k$ such that $|U_n| = 1$ and
\[
\lim_{n \to \infty} \lambda(U_n)T^{q}(U_n) = 0.
\]
Let $\O_n = (0,1)^{d-k}\times U_n$, then $|\O_n| = 1$ and, due to Proposition~\ref{propocomputedegenerate} we have
\[
\lambda_{\mathcal{E}_k}(\O_n) = \inf_{x \in \R^{d-k}}\lambda_{\mathcal{E}_k}((\O_n)_x) = \inf_{x \in (0,1)^{d-k}} \lambda(U_n) = \lambda(U_n)
\]
and
\[
T_{\mathcal{E}_k}(\O_n) = \int_{(0,1)^{d-k}}T(U_n) \ dx = T(U_n).
\]
Therefore
\[
\lim_{n \to \infty} \lambda_{\mathcal{E}_k}(\O_n)T_{\mathcal{E}_k}^q(\O_n) = \lim_{n \to \infty} \lambda(U_n)T^{q}(U_n) = 0
\]
and the conclusion follows.
\end{proof}

\section{Final remarks and questions}\label{s_open}

In this section we list some open problems that in our opinion merit to be further investigated, together with some additional comments and remarks.

\begin{itemize}

\item The first question is about the existence of optimal seminorms $H_{opt}$ for the functional $F_{q,\O}$ over the entire class $\s(\HH)$. In Section \ref{scont}, we examined several continuity properties of the mappings $H\mapsto\lambda_H(\O)$ and $H\mapsto T_H(\O)$. In particular, Example \ref{edisc} illustrates that, in the case of $\lambda_H$, additional structural assumptions on the domain $\O$ are necessary in order to ensure continuity. By contrast, the functional $T_H$ appears to exhibit a greater degree of robustness with respect to variations of the seminorm $H$. This observation naturally leads to the question of identifying general conditions on the domain $\O$ under which both mappings are continuous from $\HH$ (endowed with the topology of uniform convergence on compact subsets) into $\R$. This result would immediately yield the existence of optimal seminorms $H$ for the functional $F_{q,\O}$ over the entire class $\s(\HH)$, both in the corresponding minimization and maximization problems.

\item As $q \to 0$, the contribution of the term $T_H^q(\O)$ becomes progressively less significant, and the functional $F_{q,\O}$ approaches the limit functional $\lambda_H(\O)$. As already illustrated in Example~\ref{edisc}, the associated minimization problem for $\lambda_H(\O)$ may fail to admit a solution for general domains $\O$. It therefore remains unclear whether, for fixed but sufficiently small $q$, one can construct an analogous counterexample showing the nonexistence of minimizers for $F_{q,\O}$ as well. We believe that this delicate issue deserves further and more systematic investigation.

By contrast, the maximization problem associated with $F_{q,\O}$ always admits an optimal solution $H_{opt}$ for $q$ sufficiently small; moreover, this solution can be shown to be a norm. An interesting open question is whether this optimal norm coincides with the Euclidean one when $q$ is small enough. In our view, this issue also calls for deeper and more careful analysis.

\item At the opposite extreme, as $q\to\infty$, the influence of the term $\lambda_H(\O)$ becomes progressively negligible, and the minimization problem for $F_{q,\O}$ approaches that of the functional $T_H(\O)$. The latter problem is well understood and is uniquely minimized by the Euclidean norm. In this regime, we are able to show that, for $q$ sufficiently large, a minimizer $H_{opt}$ indeed exists and is a norm. What remains open, however, is whether this optimal norm must necessarily coincide with the Euclidean one, or whether new minimizers may arise in the large regime of parameter $q$.

Similarly to the minimization problem for small $q$, in the opposite regime of large $q$ the existence of solutions to the maximization problem for $F_{q,\O}$, for general domains $\O$, remains an open question. Establishing either positive results or suitable counterexamples in this setting appears to be a challenging and promising direction for future research.

\item Throughout this work, we have restricted our attention to energies exhibiting quadratic growth, as described in \eqref{energy}. However, a natural and substantially broader framework is obtained by considering energies with general $p$-growth (with $p>1$), namely
\[E_H(u)=\frac1p\int_\O H^p(\nabla u)\,dx.\]
It is reasonable to expect that many of the results established in the present paper extend to this more general setting, possibly under appropriate modifications of the assumptions and techniques. Such an extension could lead to further interesting developments, particularly in the limiting regimes as $p\to1$ and $p\to\infty$.

\end{itemize}

\bigskip

\noindent{\bf Acknowledgments.} GB is member of the Gruppo Nazionale per l'Analisi Matematica, la Probabilit\`a e le loro Applicazioni (GNAMPA) of the Istituto Nazionale di Alta Matematica (INdAM). RFH is supported by CAPES (PROEX 88887.712161/2022-00) and would like to thank Scuola Normale Superiore (SNS) for the exchange period and Luigi Ambrosio for connecting the authors.

\bigskip\small\noindent
Giuseppe Buttazzo: Dipartimento di Matematica, Universit\`a di Pisa\\
Largo B. Pontecorvo 5, 56127 Pisa - ITALY\\
{\tt giuseppe.buttazzo@unipi.it}\\
{\tt http://www.dm.unipi.it/pages/buttazzo/}

\bigskip\small\noindent
Raul Fernandes Horta: Departamento de Matem\'atica, Universidade Federal de Minas Gerais\\
Caixa Postal 702, 30123-970, Belo Horizonte, MG - BRAZIL\\
{\tt raul.fernandes.horta@gmail.com}

\end{document}